\documentclass[12pt,a4paper]{article}
\usepackage[english]{babel}
\usepackage{times}
\usepackage{graphicx}
\usepackage{amscd}
\usepackage{amsmath}
\usepackage{amsfonts}
\usepackage{amssymb}
\usepackage{amsthm}
\usepackage{latexsym}
\usepackage[final]{showkeys}

\newtheorem{theorem}{Theorem}
\newtheorem{proposition}[theorem]{Proposition}
\newtheorem{lemma}[theorem]{Lemma}
\newtheorem{corollary}[theorem]{Corollary}

\newcommand{\C}{\mathbb{C}}
\newcommand{\R}{\mathbb{R}}
\newcommand{\N}{\mathbb{N}}

\begin{document}

\title{Boggio's formula \\ for fractional polyharmonic Dirichlet problems}

\author{
Serena Dipierro \quad \& \quad Hans-Christoph Grunau  }
\date{} \maketitle

\begin{abstract}
Boggio's formula in balls is known for integer-polyharmonic Dirichlet problems 
and for fractional Dirichlet problems with fractional parameter less than~$1$. 
We give here a consistent formulation for fractional polyharmonic
Dirichlet problems such that Boggio's formula in balls 
yields solutions also for the general fractional case.
\end{abstract}
\renewcommand{\thefootnote}{}

\section{Introduction}

Goal of this paper is to extend Boggio's  classical formula \cite[p. 126]{boggio}
for solutions to Dirichlet problems in the unit ball $B:=B_1(0)$ of~$\R^n$
to ``polyharmonic'' operators of any fractional order~$s>0$. In this context we give
a consistent definition for $(-\Delta)^s$ when applied to functions which are 
merely in $H^s(\mathbb{R}^n)\cap H^{2s} (B)$.
For operators of order~$s\in\N$, the formula of Boggio is a classical
tool to construct solutions of the $s$-polyharmonic equation 
with right hand side equal to~$f$
and subject to homogeneous boundary data. 
\medskip

More precisely, Boggio's  formula    (see also \cite[p. 51]{GGS}) states that
the polyharmonic Green function with $s\in\N$
in the unit ball $B=B_1(0)\subset\mathbb{R}^n$
is given by
\begin{eqnarray}
G_{s} (x,y)&:=&k_{s,n} |x-y|^{2s-n} \int_1^{\left| |x|y-\frac{x}{|x|}\right|/|x-y|}
\left( v^2-1\right)^{s-1}v^{1-n}\, dv,\label{eq:Boggioformula}\\
&&{\mbox{with}}\qquad
k_{s,n}:=\frac{1}{n e_n 4^{s-1} \Gamma(s)^2},
\qquad{\mbox{and}}\qquad
e_n:=\frac{\pi^{n/2}}{\Gamma(1+n/2)}
\label{eq:Boggioformulaconstant}
\end{eqnarray}
for $x,y\in \overline{B}$, $x\not=y$. This formula showed up also in potential
theory (see also~\cite{landkof, Riesz}) for $s\in(0,1)$. 
For a probabilistic point of view, see \cite{BGR, Kul}.
Recently, the Dirichlet problem
for the fractional Laplacian attracted quite some attention 
(see e.g. \cite{bucur, pro, DPV, SV-DCDS} and the references therein)
and it was shown in
Theorems~3.1 and~3.3 in~\cite{bucur} that Boggio's formula remains valid
also for~$s\in(0,1)$. It is a quite remarkable fact that such~$G_s$ has
a rather explicit expression, in terms of the fundamental solution
and a weighted integral containing a Kelvin transformation.\medskip

To sum up, this means that, defining
$$
u(x):=\int_{B}G_{s} (x,y)f(y)\, dy,
$$
one obtains the unique solution $u$ of the $s$-polyharmonic Dirichlet problem in the unit 
ball $B$, provided that
\begin{itemize}
\item $s\in \mathbb{N}$, see \cite{boggio};
\item $s\in (0,1)$, see \cite{bucur} and the references therein. 
\end{itemize}
We shall show that Boggio's formula holds true for any $s\in (0,\infty)$ and~$n\in\N$. 
To this end we write
\begin{equation}\label{smsigma}
s=m+\sigma,\qquad{\mbox{with}}\qquad m\in\mathbb{N}\qquad{\mbox{and}}\qquad 
\sigma\in (0,1),
\end{equation}
and define for $u\in H^s(\mathbb{R}^n)\cap H^{2s} (B)$
\begin{equation}\label{operator}
(-\Delta)^s u:=(-\Delta)^m (-\Delta)^\sigma u.
\end{equation}
Here, $(-\Delta)^\sigma u$ is the so-called fractional Laplacian, for which 
we use the usual nonlocal definition given, 
for instance, in~\cite{bucur,pro,DPV,dyda}:
\begin{equation}\label{eq:def_fract_Laplacian}
\begin{split}&
(-\Delta)^\sigma u (x):= C(n,\sigma)\lim_{\varepsilon\searrow 0}
\int_{\mathbb{R}^n \setminus B_\varepsilon(x)}
\frac{u(x)-u(y)}{|x-y|^{n+2\sigma}},\\ {\mbox{where }}\;&
C(n,\sigma):=\frac{4^\sigma \Gamma(\frac{n}{2}+\sigma)}{-\Gamma(-\sigma)\, \pi^{n/2}}.
\end{split}\end{equation}
The operator~$(-\Delta)^m$ has to be understood in the classical pointwise sense. 
In this regularity class
the order of applying $(-\Delta)^\sigma$ and $(-\Delta)^m$ in (\ref{operator}) 
is quite essential.

For functions $u$, which are even in $H^{2s} (\mathbb{R}^n)$, 
the operator~$(-\Delta)^s$ may be equivalently defined by using the Fourier transform, 
that is, for any~$\xi\in\R^n$,
\begin{equation}\label{fourier}
{\mathcal{F}} \left( (-\Delta)^s u\right)(\xi) = |\xi|^{2s}\hat{u}(\xi),   
\end{equation}
where
$$ ({\mathcal{F}}u)(\xi)=\hat{u}(\xi):=\frac{1}{(2\pi)^n}\int_{\mathbb{R}^n}e^{-i\xi\cdot x} 
u(x)\, dx$$
denotes the Fourier transform of~$u$ (see~\cite[Prop. 3.3]{DPV} and also \cite{rosoton}). 

\medskip 

In order to state our main result, we denote 
by~$\alpha=(\alpha_1, \dots,\alpha_n)\in \N^n$ a multi-index and by~$|\alpha|=\alpha_1 +
\dots +\alpha_n$ the length of~$\alpha$. 

We will show that the Green function 
introduced in~\eqref{eq:Boggioformula} can be used to get a solution~$u$ to 
the $s$-polyharmonic  Dirichlet problem in the following sense:
$$
(-\Delta)^s u=f \mbox{\ in\ }B, \qquad u\in H^s(\mathbb{R}^n).
$$
More precisely, the main result that we prove in this paper is the following: 

\begin{theorem} \label{TH}
Let~$s\in(0,\infty)$ and~$f\in C^\infty_0(B)$. Set
\begin{equation}\label{eq:solution}
u(x):=\left\{ \begin{array}{ll}
      \displaystyle         \int_{B}G_{s} (x,y)f(y)\, dy\qquad &\mbox{\ for\ } x\in B,\\
\,\\
            0\qquad &\mbox{\ for\ } x\not \in B.
              \end{array}
\right.
\end{equation}
Then~$u\in H^s(\R^n)$ is the unique solution to
\begin{equation}\label{sol f}
\left\{ 
\begin{matrix}
(-\Delta)^s u=f  & \mbox{\ in\ }B, \\
u=0  & \mbox{\ in\ }\R^n\setminus B.
\end{matrix} 
\right.
\end{equation}
In addition, $u\in C^{m,\sigma} (\overline{B}) \cap C^\infty (B)$ and 
\begin{equation}\label{addition}
u(x)=(1-|x|^2)_+^s \tilde u(x)
\end{equation}
with some $\tilde{u}\in C^\infty(\R^n)$.
\end{theorem}

Once, Boggio's solution is proved to belong to 
a unique class, then it coincides with any solution
obtained by variational or functional analytic
methods and the abstract theory developed by~\cite{rosoton, SV-DCDS} applies.

We emphasise that in order to interpret \eqref{sol f} in a strong sense, the order of
differentiation as defined in \eqref{operator} is essential, i.e. 
$(-\Delta)^m ((-\Delta)^\sigma u)=f$ in $B$. Indeed, Theorem~\ref{TH} shows 
the consistency of this definition: the fractional polyharmonic operator
should be defined such that the function given by Boggio's formula yields
the unique solution of the corresponding Dirichlet problem.

As a biproduct of Theorem~\ref{TH} and the strict positivity of Boggio's 
Green function we have the following positivity preserving property:
\begin{quote}
 Let $f$ and $u$ be as in Theorem~\ref{TH}. Then $f\ge 0$, $f\not\equiv 0$ implies that
$u>0$ in $B$.
\end{quote}
It is known that such a property fails in general domains for integer $s\ge 2$, 
see the discussion
and the references in \cite{GGS}. While the present paper was submitted and
under review, Abatangelo,  Jarohs and Salda\~{n}a sent us their preprint \cite{AJS}
where precisely this question is studied for~$s>1$. Independently and (almost) simultaneously
they achieved among others the same result as in our Theorem~\ref{TH}, but their 
approach is quite different.

This type of fractional higher order operators plays an important role in analysis, 
see~\cite{sjo, samko, cots, katz, EIK, tao, grubb, miao}, 
in geometry, see~\cite{chang, GZ, jin}, and in some applications, see~\cite{magin, zhu}. 
See in particular~\cite{rosoton}, where Pohozhaev identities for higher order fractional Laplacians 
have been obtained. 
\medskip 

The Boggio-type formula obtained in Theorem~\ref{TH}
is in perfect agreement with the cases~$s\in(0,1)\cup\N$,
which were already known, but the extension to the whole
interval~$s\in(0,+\infty)$ is based on appropriate series expansions
and analytic continuation.
\medskip

The paper is organised as follows. In Section~\ref{sec:mob} we show the 
covariance under M\"obius transformations of the operator introduced in~\eqref{operator}. 
In Section~\ref{sec:green} we provide an expression for the fractional Laplacian 
of order~$\sigma\in(0,1)$ of the Green function as given by formula~\eqref{eq:Boggioformula}
with pole at the origin. This will imply that the function in~\eqref{eq:Boggioformula}
is the Green function for the operator~$(-\Delta)^s$ with pole at the origin. 
Here we strongly rely on fractional differentiation 
results due to Dyda~\cite{dyda}, which were 
developed further by him, Kuznetsov, and Kwa\'{s}nicki in the 
recent work~\cite{DKK}.

Then, by the covariance under M\"obius transformations, we will obtain the Green function 
for the fractional Laplacian of any order~$s\in(0,\infty)$ at any pole. 

In Section~\ref{sec:proof} we complete the proof of Theorem~\ref{TH}. 

\section{Covariance under M\"obius transformations}\label{sec:mob}

In this section we show that the operator introduced in~\eqref{operator} is 
covariant under M\"obius transformations. This property holds true for polyharmonic operators, 
see e.g. Lemma~6.14 in~\cite{GGS}, and for the fractional Laplacian, see e.g. Lemma~2.2 and Corollary~2.3 
in~\cite{FW} and Proposition~A.1 in~\cite{ROS}. 

The extension of these covariance properties to operators of any fractional powers
was not available in the literature (to the best of our knowledge)
and we obtain it by using complex analysis methods and unique
continuation of analytic functions, given that the desired formulas hold
in a nontrivial interval in the reals. To this aim, we use the
notation~${\mathcal{S}}$ to denote the Schwartz space of smooth
and rapidly decreasing functions and we have the following result:

\begin{lemma}\label{COMPLESSO}
Fix~$x\in\R^n$ and~$v\in {\mathcal{S}}$.
Then, the map
$$ {\mathcal{T}}_{v,x} (s):= (-\Delta)^s v(x)$$
is analytic for~$s\in (0,+\infty)$.
\end{lemma}

\begin{proof} We use complex analysis, so we will take~$s
\in\C$, with~$\Re s>0$.
We define
\begin{equation}\label{fourier2} 
w(s): = \int_{\R^n} e^{ix\cdot\xi}\, |\xi|^{2s} \,\hat v(\xi)\,d\xi .\end{equation}
We observe that, for any~$z\in\C$,
\begin{eqnarray*}
&& |e^z-1-z| =\left| \sum_{k=2}^{+\infty}\frac{z^k}{k!}\right|
\le \sum_{k=2}^{+\infty}\frac{|z|^k}{k!} 
= |z|^2 \sum_{k=2}^{+\infty}\frac{|z|^{k-2}}{k!} \\
&&\qquad \le|z|^2 \sum_{k=2}^{+\infty}\frac{|z|^{k-2}}{(k-2)!}
= |z|^2 e^{|z|}.
\end{eqnarray*}
So, for any~$h\in\C$,
we use this formula with~$z:= 2h\log|\xi|$ and we obtain
\begin{eqnarray*}
&& \Big| |\xi|^{2(s+h)}-|\xi|^{2s} -2h |\xi|^{2s}\,\log|\xi| \Big|
= |\xi|^{2\Re s}
\Big| |\xi|^{2h}-1-2h\log|\xi| \Big| \\
&&\qquad= |\xi|^{2\Re s}
\Big| e^{ 2h\log|\xi| } - 1-2h\log|\xi| \Big|
\le 4|h|^2 |\xi|^{2\Re s} \log^2 |\xi| \,e^{2|h|\,|\,\log|\xi|\,|}.
\end{eqnarray*}
Now we observe that
$$ e^{|\,\log|\xi|\,|} = \begin{cases}
|\xi| & {\mbox{ if }} |\xi|\ge 1,\\
|\xi|^{-1} & {\mbox{ if }} |\xi|<1,
\end{cases} $$
and so, if~$|h|< \Re s/2$,
$$ e^{2|h|\,|\,\log|\xi|\,|}\le |\xi|^{\Re s} +|\xi|^{-\Re s}.$$
Hence, we have found that
$$ \Big| |\xi|^{2(s+h)}-|\xi|^{2s} -2h |\xi|^{2s}\log|\xi| \Big|
\le 4|h|^2\,\big( |\xi|^{3\Re s}+ |\xi|^{\Re s} \big)\,\log^2 |\xi| .$$
In consequence of this, for any~$h\in\C$ with~$0<|h|< \Re s/2$,
\begin{eqnarray*}
&& \left| \frac{w(s+h)-w(s)}{h}
-2\int_{\R^n} e^{ix\cdot\xi}\, |\xi|^{2s}\,\log|\xi| \,\hat v(\xi)\,d\xi
\right|
\\ &&\qquad=\left| \frac{1}{h}
\int_{\R^n} e^{ix\cdot\xi}\,\hat v(\xi)\,
\Big( |\xi|^{2(s+h)}-|\xi|^{2s} -2h |\xi|^{2s}\,\log|\xi|\Big)\,d\xi\right|\\
&&\qquad \le 4|h|\,
\int_{\R^n} |\hat v(\xi)|\,
\big( |\xi|^{3\Re s}+ |\xi|^{\Re s} \big)\,\log^2 |\xi|\,d\xi\le
C\,|h|,
\end{eqnarray*}
for some~$C>0$ (possibly depending on~$n$, $s$ and~$v$).
By sending~$h\to0$, we then obtain that~$w$ is differentiable
in the complex sense, and so analytic in~$\{\Re s>0\}$,
which gives the desired result by comparing~\eqref{fourier} and~\eqref{fourier2}.
\end{proof}

With the aid of the latter result, we can now
prove the desired covariance under M\"obius transformations:

\begin{lemma}\label{le:mob}
Let~$\phi$ be a smooth M\"obius transformation 
in~$\R^n\setminus \{ x_0\}$ and let~$J_\phi$ be the modulus 
of the Jacobian determinant. 

Then, for any~$s\in(0,\infty)$ and any~$u\in C^\infty_0(\R^n\setminus\{x_0\})$, 
\begin{equation}\label{jac}
(-\Delta)^s \left( J_\phi^{\frac{1}{2}-\frac{s}{n}} u \circ \phi\right) 
= J_\phi^{\frac{1}{2}+\frac{s}{n}} \big( (-\Delta)^s u\big) \circ \phi.\end{equation} 
\end{lemma}

\begin{proof} 
We recall that any M\"obius transform~$\phi$ 
can be seen as a finite combination of similarities and inversions. 
In particular, from Corollary~4 on page~39 of~\cite{Reshe} we have that 
we can write~$\phi=\phi_1\circ j\circ \phi_2$, where~$j$ is an inversion, 
i.e~$j(x)=|x|^{-2}x$, and~$\phi_1$, $\phi_2$ are similarities, 
i.e.~$\phi_i(x)=a_i+c_iF_ix$, for~$i=1,2$, 
with~$c_i>0$, $a_i\in\R^n$ and~$F_i$ an orthogonal matrix. 
Therefore, 
\begin{equation}\label{discussion}
{\mbox{it suffices to show~\eqref{jac} for translations, rotations and inversions.}}
\end{equation}

Notice that formula~\eqref{jac} is easily verified if~$\phi$ is a translation or a rotation, 
since in these cases the Jacobian is equal to~$1$, and therefore, 
using the invariance under translations and rotations of 
polyharmonic operators and fractional Laplacians, 
\begin{eqnarray*}
&&(-\Delta)^s\left(u \circ \phi\right) = (-\Delta)^m \left[ 
(-\Delta)^\sigma\left(u \circ \phi\right)\right] 
= (-\Delta)^m \left[ \left((-\Delta)^\sigma u\right) \circ \phi\right]\\
&&\qquad =\left[ (-\Delta)^m (-\Delta)^\sigma u \right]\circ \phi 
= \left( (-\Delta)^s u\right) \circ\phi,
\end{eqnarray*}
as desired. 

Now we focus on the case in which~$\phi$ is the inversion with respect to the unit sphere.   
Namely, we consider the transformation that associates to any~$x\in\R^n\setminus \{0\}$ 
the point~$x^*:=x/|x|^2\in\R^n\setminus \{0\}$ 
and we observe that~$J_\phi(x)=|x|^{-2n}$. Hence we set
$$ 
u^*(x):=J_\phi^{\frac{1}{2}-\frac{s}{n}}(x)\, u(x^*) = \frac{1}{|x|^{n-2s}}u(x^*)
 $$
and we have that, according to our assumption, $u^*\in C^\infty_0(\R^n\setminus \{0\})$.
With this notation, we claim that
\begin{equation}\label{kelvin}
(-\Delta)^s u^*(x) = \frac{1}{|x|^{n+2s}} (-\Delta)^s u(x^*).  
\end{equation}
To prove this, 
use the notation introduced in Lemma~\ref{COMPLESSO}
to say that~\eqref{kelvin} is equivalent to
$$ F(s):={\mathcal{T}}_{u^*,x}(s) -  \frac{1}{|x|^{n+2s}} \;
{\mathcal{T}}_{u,x^*}(s)=0.$$
We remark first that, thanks to our smoothness and compact support assumptions,
$$
\hat{v} (s,\xi) :=\frac{1}{(2\pi)^n}\int_{\mathbb{R}^n}e^{-i\xi\cdot x} 
|x|^{2s-n}u\left(\frac{x}{|x|^2}\right)\, dx
$$ 
depends holomorphically on $s$ in $\Re s>0$. 
The smoothness and fast decay of $\xi\mapsto \hat{v} (s,\xi) $
is locally uniform with respect to $s$. Hence,
$F$ is analytic when~$s\in(0,+\infty)$, thanks to
Lemma~\ref{COMPLESSO}. Also, $F(s)=0$ for any~$s\in(0,1]$,
since~\eqref{kelvin} is known for this range of parameters~$s$
(see
Lemma~2.2 and Corollary~2.3 
in~\cite{FW} and Proposition~A.1 in~\cite{ROS}).
Then, by analytic unique continuation, we conclude that~$F$ vanishes
identically in~$(0,+\infty)$, which proves~\eqref{kelvin}. Then the proof of Lemma~\ref{le:mob} is also completed, 
thanks to~\eqref{discussion}.
\end{proof}

\section{Green function with pole at the origin}\label{sec:green}

In this section we focus on the case in which the pole for the Green function
is the origin. For this, we define, for~$r\in (0,\infty)$,
\begin{equation}\label{eq:aux_funct}
 \tilde{G}_s (r):=\left\{\begin{array}{ll}
  \displaystyle   r^{2s-n}\int_1^{1/r} (v^2-1)^{s-1} v^{1-n}\, dv  
\quad & \quad \mbox{\ if\ }r\in(0,1],\\[2mm]
\displaystyle    0 \quad & \quad \mbox{\ if\ }r\in(1,\infty)
                         \end{array}
\right.
\end{equation}
so that
$$
G_{s} (0,y)=k_{s,n}\tilde{G}_s (|y| ) ,\quad y\in\overline{B}\setminus \{0\}.
$$
We will prove that~$\tilde{G}_s$ (suitably renormalised) is the Green function for~$(-\Delta)^s$
with pole at the origin. 
More precisely, we prove that: 
\begin{proposition}\label{prop:pole}
For~$r=|y|\in(0,1)$ we have that
\begin{equation}\label{po1}
(-\Delta)^s \tilde{G}_s (|y|) =\big(k_{s,n}\big)^{-1} \delta_0(y), 
\end{equation}
where~$k_{s,n}$ is given by formula~\eqref{eq:Boggioformulaconstant}.
\end{proposition}

Then the general case (that is when the pole is not the origin) will follow 
from the covariance under M\"obius transformation of the operator~$(-\Delta )^s$, 
as stated in Section~\ref{sec:mob}. 
Proposition~\ref{prop:pole} gives a precise 
statement of what is outlined in a recent work of Dyda, Kuznetsov, and Kwa\'{s}nicki,
see \cite[Remark 1]{DKK}. One should, however, observe that the Green's functions
mentioned there were obtained before only for $s\in(0,1)$.
\smallskip 

In order to prove Proposition~\ref{prop:pole} 
we shall deduce a series expansion for $\tilde{G}_s$, where we can rely on 
calculations by Dyda \cite{dyda} in order to compute~$(-\Delta )^\sigma_y G_{s} (0,y)$. 
For this purpose we recall the definition of the
{\it Pochhammer} symbols:
\begin{equation}\label{eq:pochhammer}
 (a)_k:=\prod^{k-1}_{j=0} (a+j),\qquad a\in\mathbb{R}, \quad k\in \mathbb{N}_0.
\end{equation}
In what follows we perform calculations 
with hypergeometric series which are related
to those by Bucur~\cite[Section 3]{bucur} and by Dyda, Kuznetsov, and Kwa\'{s}nicki~\cite{DKK}.
We start by writing~$\tilde{G}_s$ in a useful way for our computations. 

\begin{lemma}\label{lem:series_exp_G}
 For $r\in(0,\infty)$ we have that
\begin{equation}\label{eq:series_exp_G}
 \tilde{G}_s (r) =\sum^\infty_{k=0} \frac{\left(\frac{n}{2}\right)_k}{2 (s)_{k+1}} (1-r^2)_+^{k+s}.
\end{equation}
\end{lemma}

This claim could also be deduced from \cite[Remark 1]{DKK}.

\begin{proof}[Proof of Lemma~\ref{lem:series_exp_G}]
Differentiating $\tilde{G}_s $ yields for $r\in(0,1]$:
\begin{eqnarray*}
 \tilde{G}_s ' (r) &=&(2s-n) r^{2s-n-1}\int_1^{1/r} (v^2-1)^{s-1} v^{1-n}\, dv
     -r^{2s-n-2} \left(\frac{1}{r^2} - 1 \right)^{s-1} r^{n-1}\\
&=& \frac{(2s-n)}{r} \tilde{G}_s (r) -\frac{1}{r}(1-r^2)^{s-1}.
\end{eqnarray*}
Hence $\tilde{G}_s $ solves the following initial value problem:
\begin{equation}\label{eq:ivp_G}
\begin{cases}
\tilde{G}_s ' (r)-\frac{(2s-n)}{r} \tilde{G}_s (r)= -\frac{1}{r}(1-r^2)^{s-1}, \quad r\in(0,1],\\
\tilde{G}_s (1)=0.
\end{cases}\end{equation}
A direct  calculation shows that the unique solution of \eqref{eq:ivp_G} 
is given by the right-hand side
of \eqref{eq:series_exp_G}, and this finishes the proof of Lemma~\ref{lem:series_exp_G}.
\end{proof}

Now we employ Lemma~\ref{lem:series_exp_G} and Theorem~1 in~\cite{dyda} 
to find an expression 
of~$(-\Delta)^\sigma_y \tilde{G}_s(0,y)$ (extended by~$0$ outside~$B$) 
in terms of an auxiliary function~$G^\sharp_s$ 
(given in~\eqref{eq:def_G_hat}). 

\begin{lemma}\label{lem:series_exp_G_delta_sigma}
For $r\in(0,1)$ we have that 
$$
(-\Delta)^\sigma \tilde{G}_s (r) =\frac{4^{\sigma-\frac{1}{2}} 
\Gamma(\frac{n}{2}+\sigma)\Gamma(s)}{\Gamma(\frac{n}{2}) \cdot m! }\, G^\sharp_s (r),
$$
with
\begin{equation}\label{eq:def_G_hat}
 G^\sharp_s (r)= \sum^\infty_{k=0}\frac{\left(\frac{n}{2}\right)_k }{  (m+1)_k}
                      \, {}_2F_1 \left( \frac{n}{2}+\sigma,-k-m;\frac{n}{2};r^2\right) .
\end{equation}
Here,   ${}_2F_1$ denotes Gau\ss's hypergeometric function.
\end{lemma}

\begin{proof}
We recall Euler's Beta-function $B(p,q):=\frac{\Gamma(p) \Gamma(q) }{\Gamma(p+q) }$.
 When applying Dyda's fractional differentiation result \cite[Theorem 1]{dyda} one should
observe that he uses $\Delta^\sigma=-(-\Delta)^\sigma$. This  yields immediately:
\begin{eqnarray*}
 (-\Delta)^\sigma \tilde{G}_s (r) &=&\frac{4^{\sigma-\frac{1}{2}} 
\Gamma(\frac{n}{2}+\sigma)}{\Gamma(\frac{n}{2}) \, \Gamma(-\sigma)}\,\\
&&\quad \cdot
\sum^\infty_{k=0} \frac{\left(\frac{n}{2}\right)_k B(-\sigma,k+m+\sigma+1)}{ (s)_{k+1}}\, 
{}_2F_1 \left( \frac{n}{2}+\sigma,-k-m;\frac{n}{2};r^2\right)\\
&=&\frac{4^{\sigma-\frac{1}{2}} \Gamma(\frac{n}{2}+\sigma)}{\Gamma(\frac{n}{2}) }
\sum^\infty_{k=0} \frac{\left(\frac{n}{2}\right)_k \Gamma(k+m+\sigma+1)}{ (s)_{k+1} 
\Gamma(k+m+1)} \, {}_2F_1 \left( \frac{n}{2}+\sigma,-k-m;\frac{n}{2};r^2\right)\\
&=&\frac{4^{\sigma-\frac{1}{2}} \Gamma(\frac{n}{2}+\sigma)}{\Gamma(\frac{n}{2}) }
\sum^\infty_{k=0} \frac{\left(\frac{n}{2}\right)_k \Gamma(k+s+1)}{ (s)_{k+1} (m+1)_k
\Gamma(m+1)} \, {}_2F_1 \left( \frac{n}{2}+\sigma,-k-m;\frac{n}{2};r^2\right)\\
&=&\frac{4^{\sigma-\frac{1}{2}} \Gamma(\frac{n}{2}+\sigma)}{\Gamma(\frac{n}{2})\cdot m! }
\sum^\infty_{k=0} \frac{\left(\frac{n}{2}\right)_k (s)_{k+1}\Gamma(s)}{ (s)_{k+1} 
(m+1)_k} \, {}_2F_1 \left( \frac{n}{2}+\sigma,-k-m;\frac{n}{2};r^2\right)\\
&=&\frac{4^{\sigma-\frac{1}{2}} \Gamma(\frac{n}{2}+\sigma)  \Gamma(s)}{\Gamma(
\frac{n}{2})\cdot m! }
\sum^\infty_{k=0} \frac{\left(\frac{n}{2}\right)_k }{  (m+1)_k} \, {}_2F_1 \left( 
\frac{n}{2}+\sigma,-k-m;\frac{n}{2};r^2\right),
\end{eqnarray*}
which gives the desired result. 
\end{proof}

Finally, we provide an expression for the function~$G^\sharp_s$, 
also in terms of the Green function for~$(-\Delta)^m$. 
Indeed, we have the following result:

\begin{lemma}\label{lem:G_delta_sigma_equals_polyharmonic_Green}
There exist real numbers $a_0,\ldots,a_{m-1}$ such that, for $r\in(0,1)$, we have that
\begin{equation}\label{eq:key_relation}
 G^\sharp_s (r)=\sum^{m-1}_{k=0} a_k(1-r^2)^k+
 2m \frac{\Gamma(m+\sigma) \Gamma(\frac{n}{2}) }{\Gamma(\frac{n}{2} +\sigma )
\Gamma(m)} \tilde{G}_m (r) .
\end{equation}
\end{lemma}

\begin{proof} 
If we directly insert the definition of  $\, {}_2F_1 
\left( \frac{n}{2}+\sigma,-k-m;\frac{n}{2};r^2\right)$ 
into \eqref{eq:def_G_hat}
we would come up with polynomials with terms of alternating sign. 
The whole series would then not be
absolutely convergent and one could not rearrange the order of summation. 
In order to bypass this
difficulty we use \cite[Corollary 2.3.3]{AAR}. Here it is quite important 
that~$\sigma$ is not an integer. 
We first notice that
$$
\, {}_2F_1 \left( \frac{n}{2}+\sigma,-k-m;\frac{n}{2};r^2\right)=
\frac{(-\sigma)_{k+m}}{( \frac{n}{2})_{k+m}}
\, {}_2F_1 \left(\frac{n}{2}+\sigma, -k-m; \sigma+1-k-m;1-r^2 \right),
$$
where
\begin{eqnarray*}
 \lefteqn{
\, {}_2F_1 \left(\frac{n}{2}+\sigma, -k-m; \sigma+1-k-m;1-r^2 \right)
=\sum^{k+m}_{\ell=0}\frac{\left(\frac{n}{2}+\sigma \right)_\ell 
\left( -k-m\right)_\ell}{\ell !\, \left(\sigma+1-k-m \right)_\ell} (1-r^2 )^\ell
}
\\
&=&\sum^{k+m}_{\ell=0}\frac{\left(\frac{n}{2}+\sigma \right)_\ell 
\left( m+k\right)\cdot \ldots \cdot (m+k-\ell+1)}{
\ell !\, (m+k-\sigma -1 )\cdot \ldots \cdot (m+k-\sigma -\ell)} (1-r^2 )^\ell \\
&=&\sum^{k+m}_{\ell=0}\frac{\left(\frac{n}{2}+\sigma \right)_\ell \, (m+k-\ell+1)_\ell }{
\ell ! \, (m+k-\sigma -\ell)_\ell } (1-r^2 )^\ell. \hspace*{4cm}
\end{eqnarray*}
Here one has that all summands but for $\ell=k+m$ are positive. 
This means that in what follows we
have absolute convergence and may rearrange the order of summation. 
In view of Lemma~\ref{lem:series_exp_G_delta_sigma}
we obtain that, for suitable real numbers $a_\ell$,
\begin{eqnarray*}
 G^\sharp_s (r) & = & \sum^\infty_{k=0}\frac{ (\frac{n}{2} )_k \,(-\sigma)_{k+m} }{  
(m+1)_k \,( \frac{n}{2})_{k+m} }
\sum^{k+m}_{\ell=0}\frac{ (\frac{n}{2}+\sigma   )_\ell \, (m+k-\ell+1)_\ell }{
\ell ! \, (m+k-\sigma -\ell)_\ell } (1-r^2 )^\ell\\
&=&\sum^\infty_{k=0}\frac{ (-\sigma)_{k+m} }{  (m+1)_k \,( \frac{n}{2}+k)_{m} }
\sum^{k+m}_{\ell=0}\frac{(\frac{n}{2}+\sigma )_\ell \, (m+k-\ell+1)_\ell }{
\ell ! \, (m+k-\sigma -\ell)_\ell } (1-r^2 )^\ell\\
&=&\sum^\infty_{k=m}\frac{ (-\sigma)_{k} }{  (m+1)_{ k -m} \,( \frac{n}{2}+k-m)_{m} }
\sum^{k}_{\ell=0}\frac{ (\frac{n}{2}+\sigma  )_\ell \, (k+1-\ell )_\ell }{
\ell ! \, (k-\sigma -\ell)_\ell } (1-r^2 )^\ell\\
&=& \sum^{m-1}_{\ell=0} a_\ell\, (1-r^2)^\ell\\
&&+\sum^\infty_{\ell=m} \frac{(\frac{n}{2}+\sigma )_\ell}{\ell ! }(1-r^2 )^\ell
\, \sum^\infty_{k=\ell}\frac{(-\sigma)_{k} \, (k+1-\ell )_\ell }{ (m+1)_{ k -m} 
\,( \frac{n}{2}+k-m)_{m}\,  (k-\sigma -\ell)_\ell }.
\end{eqnarray*}
Interchanging $k$ and $\ell$ leads us to 
\begin{equation}\begin{split}\label{eq:star1}
 G^\sharp_s (r) =\,& \sum^{m-1}_{k=0} a_k\, (1-r^2)^k \\
&+\sum^\infty_{k=m} \frac{(\frac{n}{2}+\sigma )_k}{k! }(1-r^2 )^k
\, \sum^\infty_{\ell=k}\frac{(-\sigma)_{\ell} \, (\ell+1- k)_k }{ (m+1)_{ \ell -m} 
\,( \frac{n}{2}+\ell-m)_{m}\,  (\ell-\sigma -k)_k } \\
=\,& \sum^{m-1}_{k=0} a_k\, (1-r^2)^k \\
&+\sum^\infty_{k=m} \frac{(\frac{n}{2}+\sigma )_k}{k! }(1-r^2 )^k
\, \sum^\infty_{\ell=0}\frac{(-\sigma)_{\ell+k} \, (\ell+1)_{ k} }{ (m+1)_{ \ell+k -m} 
\,( \frac{n}{2}+\ell +k -m)_{m}\,  (\ell-\sigma )_{ k} } \\
=\,& \sum^{m-1}_{k=0} a_k\, (1-r^2)^k \\
&+\sum^\infty_{k=m} \frac{(\frac{n}{2}+\sigma )_k}{k! \, (m+1)_{ k -m}}(1-r^2 )^k
\, \sum^\infty_{\ell=0}\frac{(-\sigma)_{\ell} \, (\ell+1)_{ k}\, \cdot \, \ell!  }{ 
(k+1)_{ \ell } \,( \frac{n}{2}+\ell +k -m)_{m} \, \cdot \, \ell! } \\
=\,& \sum^{m-1}_{k=0} a_k\, (1-r^2)^k
+\sum^\infty_{k=m} \frac{(\frac{n}{2}+\sigma )_k}{  (m+1)_{ k -m}}(1-r^2 )^k
\, \sum^\infty_{\ell=0}\frac{(-\sigma)_{\ell}  }{ \ell!   ( \frac{n}{2}+\ell +k -m)_{m}  }.
\end{split}\end{equation}
We show now that for $k\ge m$
\begin{equation}\label{eq:star2}
\frac{(\frac{n}{2}+k-m)_\ell}{(\frac{n}{2}+k)_\ell}
=\frac{(\frac{n}{2}+k-m)_m}{(\frac{n}{2}+k+\ell-m)_m}.
\end{equation}                                                     
Indeed, the claim is obvious for $\ell=m$, while for $\ell >m$ terms cancel as follows:
$$
\frac{(\frac{n}{2}+k-m)_\ell}{(\frac{n}{2}+k)_\ell}
=\frac{(\frac{n}{2}+k-m)\cdot\ldots \cdot(\frac{n}{2}+k) \cdot\ldots 
\cdot (\frac{n}{2}+k+\ell-m-1)}{(\frac{n}{2}+k) \cdot\ldots \cdot (\frac{n}{2}+k+\ell -1) }
=\frac{(\frac{n}{2}+k-m)_m}{(\frac{n}{2}+k+\ell-m)_m}.
$$
On the other hand, if $\ell<m$ we see that
$$
\frac{(\frac{n}{2}+k-m)_m}{(\frac{n}{2}+k+\ell-m)_m}
=\frac{(\frac{n}{2}+k-m)\cdot \ldots\cdot (\frac{n}{2}+k-1)}{
(\frac{n}{2}+k+\ell-m)\cdot \ldots\cdot ( \frac{n}{2}+k) 
\cdot \ldots\cdot(\frac{n}{2}+k+\ell -1) }.
=\frac{(\frac{n}{2}+k-m)_\ell}{(\frac{n}{2}+k)_\ell}.
$$
This completes the proof of~\eqref{eq:star2}.

Combining \eqref{eq:star1} and \eqref{eq:star2}
we conclude further that
\begin{eqnarray*}
 G^\sharp_s (r) & = &
\sum^{m-1}_{k=0} a_k\, (1-r^2)^k\\
&&+\sum^\infty_{k=m} \frac{(\frac{n}{2}+\sigma )_k}{  (m+1)_{ k -m} 
\,(\frac{n}{2}+k-m)_m }(1-r^2 )^k
\, \sum^\infty_{\ell=0}\frac{(-\sigma)_{\ell}  \, (\frac{n}{2}+k-m)_\ell}{ \ell!   
\,(\frac{n}{2}+k)_\ell }\\
&=&\sum^{m-1}_{k=0} a_k\, (1-r^2)^k\\
&&+\sum^\infty_{k=m} \frac{(\frac{n}{2}+\sigma )_k}{  (m+1)_{ k -m} \,
(\frac{n}{2}+k-m)_m }(1-r^2 )^k
 \, {}_2F_1 (-\sigma, \frac{n}{2}+k-m;\frac{n}{2}+k;1).
\end{eqnarray*}
We apply now Gau\ss's summation theorem, 
see \cite[Theorem 2.2.2]{AAR}, and obtain further:
\begin{eqnarray*}
  G^\sharp_s (r) & = &
\sum^{m-1}_{k=0} a_k\, (1-r^2)^k
+\sum^\infty_{k=m} \frac{(\frac{n}{2}+\sigma )_k \Gamma(\frac{n}{2}+k)\, 
\Gamma(m+\sigma)
}{  (m+1)_{ k -m} \,(\frac{n}{2}+k-m)_m \Gamma(\frac{n}{2}+k+\sigma)\, 
\Gamma(m)}(1-r^2 )^k\\
&=&\sum^{m-1}_{k=0} a_k\, (1-r^2)^k
+\frac{\Gamma(m+\sigma)\Gamma(\frac{n}{2})}{\Gamma(\frac{n}{2}+\sigma)\Gamma(m)}
\sum^\infty_{k=m} \frac{ (\frac{n}{2})_k\, 
}{  (m+1)_{ k -m} \,(\frac{n}{2}+k-m)_m }(1-r^2 )^k\\
&=&\sum^{m-1}_{k=0} a_k\, (1-r^2)^k
+\frac{\Gamma(m+\sigma)\Gamma(\frac{n}{2})}{\Gamma(\frac{n}{2}+\sigma)\Gamma(m)}
\sum^\infty_{k=m} \frac{ (\frac{n}{2})_{k-m}\, 
}{  (m+1)_{ k -m}  }(1-r^2 )^k\\
&=&\sum^{m-1}_{k=0} a_k\, (1-r^2)^k
+2m\frac{\Gamma(m+\sigma)\Gamma(\frac{n}{2})}{
\Gamma(\frac{n}{2}+\sigma)\Gamma(m)}
\sum^\infty_{k=0} \frac{ (\frac{n}{2})_{k}\, 
}{ 2 (m)_{ k+1 }  }(1-r^2 )^{k+m}.
\end{eqnarray*}
Thus, we recall Lemma~\ref{lem:series_exp_G} and end up with
$$
 G^\sharp_s (r) =\sum^{m-1}_{k=0} a_k\, (1-r^2)^k
+2m\frac{\Gamma(m+\sigma)\Gamma(\frac{n}{2})}{
\Gamma(\frac{n}{2}+\sigma)\Gamma(m)}
\tilde{G}_m (r).
$$
This concludes the proof of Lemma~\ref{lem:G_delta_sigma_equals_polyharmonic_Green}.
\end{proof}

With all this in hand, we are able to complete the proof of Proposition~\ref{prop:pole}. 

\begin{proof}[Proof of Proposition~\ref{prop:pole}] 
Recalling the definition of~$(-\Delta)^s$ in~\eqref{operator} and 
combining Lemmata~\ref{lem:series_exp_G_delta_sigma} 
and~\ref{lem:G_delta_sigma_equals_polyharmonic_Green} we obtain 
by writing $r=|y|$ that 
\begin{eqnarray*}
(-\Delta)^s \tilde{G}_s(r) &=& 
(-\Delta)^m\, (-\Delta)^\sigma \tilde{G}_s (r) \\&=& 
\frac{4^{\sigma-\frac{1}{2}} 
\Gamma(\frac{n}{2}+\sigma)\Gamma(s)}{\Gamma(\frac{n}{2}) \cdot m! }\, 
(-\Delta)^m G^\sharp_s (r)\\
&=& \frac{4^{\sigma-\frac{1}{2}} 
\Gamma(\frac{n}{2}+\sigma)\Gamma(s)}{\Gamma(\frac{n}{2}) \cdot m! }\cdot 
2m \, \frac{\Gamma(m+\sigma) \Gamma(\frac{n}{2}) }{
\Gamma(\frac{n}{2} +\sigma )\Gamma(m)} 
(-\Delta)^m \tilde{G}_m (r)\\
&=& \frac{4^{\sigma-\frac{1}{2}} 
\Gamma(\frac{n}{2}+\sigma)\Gamma(s)}{\Gamma(\frac{n}{2}) \cdot m! }\cdot 
2m\, \frac{\Gamma(m+\sigma) \Gamma(\frac{n}{2}) }{
\Gamma(\frac{n}{2} +\sigma )\Gamma(m)} \cdot 
\Big(n e_n 4^{m-1}\Gamma(m)^2\Big)\, \delta_0(y). 
\end{eqnarray*}
Now we recall that~$s=m+\sigma$ and that~$\Gamma(m)=(m-1)!$, therefore
$$ 
(-\Delta)^s \tilde{G}_s(|y|) = n e_n 4^{s-1} \Gamma(s)^2 \delta_0(y),
$$
which gives the desired result, thanks to~\eqref{eq:Boggioformulaconstant}.
\end{proof}

\section{Proof of Theorem~\ref{TH}}\label{sec:proof}

In this section we complete the proof of Theorem~\ref{TH}, with the aid 
of the results on the Green function obtained in 
Sections~\ref{sec:mob} and~\ref{sec:green}. 
To this aim,
in the following result, we show,
by straightening the boundary of the unit ball,
that $x\mapsto (1-|x|^2)_+^\sigma\in H^\sigma(\R^n)$
(from this, in Corollary~\ref{TRE} we will obtain that
$x\mapsto (1-|x|^2)_+^s\in H^s(\R^n)$).

\begin{lemma}\label{Hs1}
For any~$\sigma\in(0,1)$, the 
function~$\R^n\ni x\mapsto (1-|x|^2)_+^{\sigma}$
belongs to~$H^{\sigma}(\R^n)$.
\end{lemma}

\begin{proof} 
With a covering argument, we may focus on proving that~$u(x)
:=(1-|x|^2)_+^\sigma$ belongs to~$H^\sigma(B_{1/10}(-e_n))$.
We use the notation~$x=(x',x_n)\in\R^{n-1}\times\R$ and,
for any~$|x|\le1$, we define
$$ y=(y',y_n)=\Phi(x):=\big( x',\,x_n+\sqrt{1-|x'|^2}\big).$$
Notice that this map sends the unit ball~$B$ into~$\R^{n-1}\times(0,+\infty)$
and we can consider the inverse map
$$ x=\Psi(y):=\big( y',\,y_n-\sqrt{1-|y'|^2}\big).$$
Notice that~$\Phi$ is smooth in~$B_{1/10}(-e_n)$, hence $\Psi$ is also
smooth on the image of~$B_{1/10}(-e_n)$, which is in turn contained into~$
D:=B^{n-1}_{1/10}\times\left[-\frac{11}{10},\,\frac{11}{10}\right]$, where the index~$n-1$ 
indicates that this  is an $(n-1)$-dimensional  ball. Hence, we define
$$ 
v(y):= u(\Psi(y))= (  2y_n \sqrt{1-|y'|^2} -y_n^2 )_+^\sigma 
$$
and we aim at showing that~$v\in H^\sigma(D)$.
Notice also that, if~$y\in D$,
$$ 
\sqrt{1-|y'|^2} \in \left[ \sqrt{\frac{99}{100}}, \;1\right],
$$
and so
$$ 
2 \sqrt{1-|y'|^2} - y_n \ge 2\sqrt{\frac{99}{100}}-\frac{11}{10}>0.
$$
Consequently, $v(y)=0$ for~$y\in D$ if and only if~$y_n<0$, and so we can write
$$ v(y)= (  2y_n \sqrt{1-|y'|^2} - y_n^2 )_+^\sigma=
(y_n)_+^\sigma \,(  2\sqrt{1-|y'|^2} - y_n )^\sigma=(y_n)_+^\sigma \,w(y),$$
with~$w$ smooth in~$D$.

As a consequence, we only need to show that~$\zeta(y):=(y_n)_+^\sigma$
belongs to~$H^\sigma(D)$. This follows from the  computation below:
\begin{eqnarray*}
&& \int_{B^{n-1}_{1/10}} \,dx'\,
\int_{-11/10}^{11/10} \,dx_n\,
\int_{B^{n-1}_{1/10}} \,dy'\,
\int_{-11/10}^{11/10} \,dy_n\,
\frac{|\zeta(x)-\zeta(y)|^2}{|x-y|^{n+2\sigma}} \\
&=& 2\int_{B^{n-1}_{1/10}} \,dx'\,
\int_{-11/10}^{11/10} \,dx_n\,
\int_{B^{n-1}_{1/10}} \,dy'\,
\int_{-11/10}^{x_n} \,dy_n\,
\frac{|(x_n)_+^\sigma-(y_n)_+^\sigma|^2}{|x-y|^{n+2\sigma}} 
\\ &=&
2\int_{B^{n-1}_{1/10}} \,dx'\,
\int_{-11/10}^{11/10} \,dx_n\,
\int_{B^{n-1}_{1/10}} \,dy'\,
\int_{-11/10}^{x_n} \,dy_n\,
\frac{|(x_n)_+^\sigma-(y_n)_+^\sigma|^2}{
|x_n-y_n|^{n+2\sigma} \left(1+\frac{|x'-y'|^2}{|x_n-y_n|^2}
\right)^{\frac{n+2\sigma}{2}}}\\
&\le& 
2\int_{B^{n-1}_{1/10}} \,dx'\,
\int_{-11/10}^{11/10} \,dx_n\,
\int_{\R^{n-1}} \,d\mu'\,
\int_{-11/10}^{x_n} \,dy_n\,
\frac{|(x_n)_+^\sigma-(y_n)_+^\sigma|^2}{
|x_n-y_n|^{1+2\sigma} \left(1+|\mu'|^2\right)^{\frac{n+2\sigma}{2}}}
\\ &\le&
C\,
\int_{-11/10}^{11/10} \,dx_n\,
\int_{-11/10}^{x_n} \,dy_n\,
\frac{|(x_n)_+^\sigma-(y_n)_+^\sigma|^2}{
|x_n-y_n|^{1+2\sigma}}
\\ &=&
C\,
\int_{0}^{11/10} \,dx_n\,
\int_{-11/10}^{0} \,dy_n\,
\frac{x_n^{2\sigma}}{
|x_n-y_n|^{1+2\sigma}} 
+
C\,
\int_{0}^{11/10} \,dx_n\,
\int_{0}^{x_n} \,dy_n\,
\frac{|x_n^\sigma-y_n^\sigma|^2}{
|x_n-y_n|^{1+2\sigma}} 
\\ &=&
C
+
C\,
\int_{0}^{11/10} \,dx_n\,
\int_{0}^{1} \,d\tau\,x_n
\frac{x_n^{2\sigma}(1-\tau^\sigma)^2}{
x_n^{1+2\sigma}(1-\tau)^{1+2\sigma}} 
\\ &=& C,
\end{eqnarray*}
for some~$C>0$, possibly varying from line to line.
\end{proof}

As a consequence of Lemma~\ref{Hs1}, we obtain

\begin{corollary}\label{TRE}
For any~$m\in\N$ and~$\sigma\in(0,1)$, the 
function~$\R^n\ni x\mapsto (1-|x|^2)_+^{m+\sigma}$
belongs to~$H^{m+\sigma}(\R^n)$.
\end{corollary}

\begin{proof} Let~$u(x):=(1-|x|^2)_+^{m+\sigma}$
and~$\alpha\in\N^n$ with~$\alpha_1+\dots+\alpha_n=m$.
By iterating Lemma~\ref{Hs1}, we obtain that~$D^\alpha u\in H^\sigma(\R^n)$.
That is, using the equivalent norm in Fourier space,
$$ +\infty>\int_{\R^n} |\xi|^{2\sigma} \,\big|{\mathcal{F}}(D^\alpha u)\,(\xi)\big|^2\,d\xi
= (2\pi)^m \,\int_{\R^n} |\xi|^{2\sigma} |\xi_1|^{2\alpha_1}\,\dots\,|\xi_n|^{2\alpha_n}\,
|\hat u(\xi)|^2\,d\xi.$$
Choosing~$\alpha=me_j$, with~$j\in\{1,\dots,N\}$, we find that
$$ +\infty>\sum_{j=1}^n
\int_{\R^n} |\xi|^{2\sigma} |\xi_j|^{2m}\,
|\hat u(\xi)|^2\,d\xi \ge \frac{1}{n^m}
\int_{\R^n} |\xi|^{2\sigma} |\xi|^{2m}\,
|\hat u(\xi)|^2\,d\xi,$$
that is the desired result.
\end{proof}

With this, we are now in the position of completing the proof of
Theorem~\ref{TH}.

\begin{proof} [Proof of Theorem~\ref{TH}] 
We observe that
\begin{equation} \label{eq:LAK:LA}
(-\Delta)^s_y G_s(x,y) = \delta_x(y), \quad (-\Delta)^s_x G_s(x,y) = \delta_y(x),\quad {\mbox{ for any }} x,y\in B,
\end{equation}
in the distributional sense.
To prove this, one can focus on the proof of the first claim, since the second is
equivalent to that, due to the symmetry of~$G_s$ under the exchange of~$x$
and~$y$. Also, the case~$x=0$ has already been considered in Proposition~\ref{prop:pole},
so we can suppose that~$x\not=0$. Hence, we can consider the M\"obius transformation
$$
\phi_x:\overline{B} \to \overline{B}, \quad \phi_x(y):= \frac{1}{|x|^2}
\left( x + (1-|x|^2) \frac{y-\frac{x}{|x|^2}}{\left| y - \frac{x}{|x|^2}\right|^2}\right),
$$ 
which is an automorphism of $\overline{B}$ and 
satisfies $\phi_x(0)=x$, $\phi_x(x)=0$ and~$\phi_x \circ \phi_x=id_{\overline{B}}$. 
Let also~$\eta\in C^\infty_0(B)$ and
$$\tilde\eta(y):=
J_{\phi_x}^{\frac{1}{2}-\frac{s}{n}}(y)\; \eta( \phi_x(y)),$$
where
$$
J_{\phi_x}(y)=(1-|x|^2)^{n}\left||x|y-\frac{x}{|x|} \right|^{-2n}
$$
is the modulus of the Jacobian determinant of~$\phi_x$. Hence,
$$
\tilde\eta(y) = (1-|x|^2)^{(n-2s)/2}\left||x|y-\frac{x}{|x|} \right|^{2s-n} \eta( \phi_x(y)).
$$
We point out that the center of inversion~$x/|x|^2$ is outside $\overline{B}$,
hence~$\phi_x\in C^\infty(\overline{B})$, and therefore~$
\tilde\eta\in C^\infty_0(B)$. Proposition~\ref{prop:pole} shows that 
$$
\tilde\eta(0) =\int_B G_s(0,y) (-\Delta)^s \tilde\eta(y)\, dy.
$$
Therefore, using Lemma~\ref{le:mob} and the fact that $J_{\phi_x}(\phi_x(y))=
\frac{1}{J_{\phi_x}(y)}$ we find that
\begin{eqnarray*}
(1-|x|^2)^{(n-2s)/2} \eta(x) &=& \tilde\eta(0)\\
&=& \int_B G_s(0,y)\, J_{\phi_x}^{\frac{n+2s}{2n}}(y) 
\left((-\Delta)^s \eta \right)\, \left(  \phi_x (y)\right) \, dy\\
&=& \int_B G_s(0,\phi_x(y))\, J_{\phi_x}^{\frac{n+2s}{2n}}(\phi_x(y))\, 
J_{\phi_x}(y) \, \left((-\Delta)^s \eta \right)(y)\, dy\\
&=& \int_B G_s(0,\phi_x(y))\, J_{\phi_x}^{\frac{n-2s}{2n}}(y)\, 
\left((-\Delta)^s \eta \right)(y)\, dy\\
&=& (1-|x|^2)^{(n-2s)/2}\int_B\left||x|y-\frac{x}{|x|} \right|^{2s-n}\, 
G_s(0,\phi_x(y))\,\left((-\Delta)^s \eta \right)(y)\, dy,
\end{eqnarray*}
hence
\begin{equation}\label{eq:Green_intermediate}
 \eta(x) = \int_B\left||x|y-\frac{x}{|x|} \right|^{2s-n}\, 
G_s(0,\phi_x(y))\,\left((-\Delta)^s \eta \right)(y)\, dy.
\end{equation}
Now, using that
$$
| \phi_x(y) | =\frac{|x-y| }{\left||x|y-\frac{x}{|x|} \right|}
$$
and 
$$
G_s(0,\phi_x(y))=k_{s,n}| \phi_x(y) |^{2s-n} \int_1^{1/| \phi_x(y) | } (v^2-1)^{s-1}v^{1-n}\, dv,
$$
we find that
\begin{eqnarray*}
\left||x|y-\frac{x}{|x|} \right|^{2s-n}\, G_s(0,\phi_x(y))&=&k_{s,n}|x-y|^{2s-n}
\int_1^{ \left||x|y-\frac{x}{|x|} \right|/|x-y|} (v^2-1)^{s-1}v^{1-n}\, dv\\
&=& G_s(x,y).
\end{eqnarray*}
This, together with \eqref{eq:Green_intermediate}, 
implies that, for any $\eta\in C^\infty_0(B)$ and any $x\in B$,
$$
\eta(x) = \int_B G_s(x,y)\, \left((-\Delta)^s \eta \right)(y)\, dy.
$$
This completes the proof of~\eqref{eq:LAK:LA}.

Now we take~$f\in C^\infty_0 (B)$ and $u$ as in~\eqref{eq:solution}. 
We show first that
\begin{equation}\label{P:we}
{\mbox{$u$ is a weak solution to~\eqref{sol f}.}}\end{equation}
To this end let $\varphi\in C^\infty_0 (B)$ and observe that 
$$
(-\Delta)^s \varphi=(-\Delta)^m (-\Delta)^\sigma \varphi 
=(-\Delta)^\sigma (-\Delta)^m\varphi .
$$
With the help of this and~\eqref{eq:LAK:LA} we conclude that
\begin{eqnarray*}
 \int_B u(x) (-\Delta)^s \varphi (x)\, dx &=& \int_B\int_B G_s(x,y) 
f(y)(-\Delta)_x^\sigma (-\Delta)_x^m  \varphi (x)\, dy\,dx\\
&=& \int_B \left( \int_B  (-\Delta)_x^\sigma G_s(x,y) \,  
(-\Delta)_x^m  \varphi (x)\, dx \right)\, f(y) \, dy\\
&=& \int_B\left( \int_B\delta_y(x) \varphi(x)\, dx \right)\, f(y) \, dy\\
&=& \int_B f(y) \varphi(y)\, dy,
\end{eqnarray*}
which proves~\eqref{P:we}.

We next show that $u$, extended outside $B$ by $0$, satisfies
\begin{equation}\label{TPR5}
u\in H^s(\mathbb{R}^n) \cap C^{m,\sigma} (\mathbb{R}^n)
\cap C^\infty (B)\end{equation} 
and so, in particular, $u$ vanishes of order $m$ on $\partial B$.

To prove~\eqref{TPR5},
we observe that Boggio's formula can be written as
\begin{equation}\label{def G-qw-qw}
G_s(x,y)=k_{s,n}\, |x-y|^{2s-n}\cdot (h\circ g)(x,y)
\end{equation}
with
\begin{equation}\label{def g}
g(x,y):= \frac{\left| |x|y-\frac{x}{|x|}\right| }{|x-y|}=\sqrt{1+\frac{(1-|x|^2)(1-|y|^2)}{|x-y|^2} }
\end{equation}
and
\begin{equation}\label{Def di h}
h(t):=\int_0^t (v^2-1)^{s-1}_+\, v^{1-n}\, dv.
\end{equation}
Notice that
$$
h(t)=\frac{1}{2} (t^2-1)_+^s \int^1_0 \tau^{s-1} \left(1+(t^2-1)\tau)^{-n/2} \right)\, d\tau
$$
where
$$
(0,\infty)\ni t\mapsto \frac{1}{2}\int^1_0 \tau^{s-1} \left(1+(t^2-1)\tau)^{-n/2} \right)\, d\tau
$$
is smooth and can be modified on $[0,1)$ to a $C^\infty $-smooth function
which is identically $0$ for $t$ close to $0$. Hence we may write
\begin{equation}\label{eq:structure_h}
 h(t)= (t^2-1)_+^s\cdot \tilde{h} (t), \quad {\mbox{ with }}  \tilde{h}\in C^\infty ([0,\infty))
\quad {\mbox{ and }}\quad  \tilde{h}|_{[0,1/2]}=0.
\end{equation}

Furthermore, since we assume that $f\in C^\infty_0(B)$, 
we may find some $\delta >0$ such that $f\in
C^\infty_0(B_{1-\delta}(0))$. Hence, 
from now on,
we take
$$
y\in B_{1-\delta}(0).
$$
As in~\cite[Lemma 4.4]{GGS}, 
we consider the sets
$$
A:= \overline{B}\times B_{1-\delta}(0), \quad 
C:=\overline{B}\times\overline{B}, \quad 
D:= \{(x,x): x\in \overline{B}\}, \quad 
$$
$$
F:=C\cap \left\{\left| |x|y-\frac{x}{|x|}\right|\le 3|x-y| \right\}, \quad 
J:=C\cap \left\{\left| |x|y-\frac{x}{|x|}\right|\ge 2 |x-y| \right\}
$$
and we define~$d(x):=\operatorname{dist} (x,\partial B)$.

We take into account first the case in which~$(x,y) \in A\cap F$. 
According to~\cite[Lemma 4.4]{GGS}, we have that
$$
d(x)=1-|x| \le c |x-y|\quad{\mbox{ and }} \quad d(y)\le c |x-y|,
$$
for some~$c\ge1$.
Since  $d(y)\ge \delta$, it follows here that \begin{equation}\label{ghA1}
|x-y|\ge \frac{\delta}{c}.\end{equation}
Accordingly, by~\eqref{def g}, we see that
\begin{equation}\label{de g sar:2}
g\in C^\infty (A\cap F)
\end{equation}
and all its derivatives are bounded in~$A\cap F$. Analogously,
\begin{equation}\label{THES}
 {\mbox{the function~$A\cap F\ni (x,y)\mapsto |x-y|^{2s-n}$
belongs to $C^\infty(A\cap F)$.}}
\end{equation}

Using this  and~\eqref{eq:structure_h}, we can write in~$A\cap F$
\begin{equation}\label{87uJAJJJJA}
\begin{split}
G_s(x,y)\;&=k_{s,n}\, |x-y|^{2s-n}\; h(g(x,y))\\
&= k_{s,n}\, |x-y|^{2s-n}\;
\left((1-|x|^2)_+^s   (1-|y|^2)^s |x-y|^{-2s} \cdot \tilde{h}(g(x,y)) \right)\\
&=  (1-|x|^2)_+^s\; g^\sharp(x,y),\end{split}\end{equation}
with some  $g^\sharp\in C^\infty(A\cap F)$.

We consider next the case in which~$(x,y) \in  A\cap (J\setminus  D)$. 
Here, we have that
$\frac{1}{c} d(y)\le d(x)\le c d(y)$, for some~$c\ge1$,
so that in particular 
\begin{equation}\label{eq:important}
d(x)\ge \frac{\delta}{c}>0
\end{equation}
and $\left| |x|y-\frac{x}{|x|} \right|\ge 1-|x|\, |y| \ge \frac{1}{c}$. 
As a consequence, in this case we may rewrite Boggio's formula
as follows:
\begin{eqnarray*}
 &&G_s (x,y)\\ &=& k_{s,n} |x-y|^{2s-n} \int_1^2
\left( v^2-1\right)^{s-1}v^{1-n}\, dv\\
&&\qquad + k_{s,n} |x-y|^{2s-n} \int_2^{\left| |x|y-\frac{x}{|x|}\right|/|x-y|}
\left( v^2-1\right)^{s-1}v^{1-n}\, dv\\
&=& c_1 |x-y|^{2s-n} + k_{s,n} |x-y|^{2s-n} \int_2^{\left| |x|y-\frac{x}{|x|}\right|/|x-y|}
v^{2s-1-n} \left( 1-\frac{1}{v^2}\right)^{s-1}\, dv\\
&=& c_1 |x-y|^{2s-n} + k_{s,n} |x-y|^{2s-n} 
\sum^\infty_{k=0}\int_2^{\left| |x|y-\frac{x}{|x|}\right|/|x-y|}
(-1)^k \left( s-1\atop k\right)v^{2s-2k-1-n}\, dv,
\end{eqnarray*}
for some~$c_1>0$.
Observe that on $A\cap J$ one has local uniform convergence of the second summand
and all its derivatives, and this justifies the exchange in 
the order of integration and summation
that we have performed here above. Hence, we can integrate the powers of~$v$
in the previous formula and end up with
$$
G_s (x,y)=c_2 |x-y|^{2s-n} +c_3\sum^\infty_{k=0}(-1)^k \left( s-1\atop k\right)\frac{1}{2s-n-2k} 
\left| |x|y -\frac{x}{|x|}\right|^{2s-2k-n}|x-y|^{2k} ,
$$
for some~$c_2$, $c_3 \in \R$. 
Only in case that $2s-2k-n=0$ (which may occur only if $s$ is a multiple of $\frac{1}{2}$)
the corresponding summand has to be replaced by 
a multiple of $\log\left(\frac{\left| |x|y -\frac{x}{|x|}\right|}{|x-y|} \right) |x-y|^{2k}$.
In consequence of this, we conclude that
in this case~$G_s (x,y)$ is the sum of~$c_2 |x-y|^{2s-n}$ 
and an analytic function, that we denote by~$H^\star(x,y)$, 
provided that $s-\frac{n}{2}\not\in \mathbb{N}_0$.
Notice that if $s-\frac{n}{2} \in \mathbb{N}_0$ 
we have a logarithmic singularity instead, which is treated analogously.

That is, we can write
\begin{equation}\label{0oKA} G_s(x,y)=c_2 |x-y|^{2s-n} + H^\star(x,y),\end{equation}
with~$H^\star(x,y)\in C^\infty(A\cap J)$.

We also set
\begin{eqnarray*} &&
{\mathcal{D}}_{1,x}:=\{ y\in B_{1-\delta}(0) {\mbox{ s.t. }} (x,y)\in A\cap F\}\\
{\mbox{and }}&&
{\mathcal{D}}_{2,x}:=\{ y\in B_{1-\delta}(0) {\mbox{ s.t. }} (x,y)\not\in A\cap F\}
.\end{eqnarray*}
With the help of (\ref{eq:important}) we see that ${\mathcal{D}}_{2,x}=\emptyset$
for $d(x)$ close to $0$.

Using this notation, we have that
\begin{eqnarray*}
&& \int_{{\mathcal{D}}_{2,x}}|x-y|^{2s-n} \;f(y)\,dy\\
&=& \int_{B_{1-\delta}(0)}|x-y|^{2s-n} \;f(y)\,dy-
\int_{B_{1-\delta}(0)\setminus{\mathcal{D}}_{2,x}}|x-y|^{2s-n} \;f(y)\,dy
\\ &=& \int_{\R^n}|x-y|^{2s-n} \;f(y)\,dy+Z(x)\\
&=& \int_{\R^n}|y|^{2s-n} \;f(x-y)\,dy+Z(x)\\
&=& Z^\star(x),
\end{eqnarray*}
for suitable~$Z$, $Z^\star\in C^\infty(\R^n)$, thanks to~\eqref{THES}.

{F}rom this, \eqref{87uJAJJJJA}
and~\eqref{0oKA}, we find that
\begin{eqnarray*}
u(x) &=& \int_{{\mathcal{D}}_{1,x}} G_s(x,y)\;f(y)\,dy+
\int_{{\mathcal{D}}_{2,x}} G_s(x,y)\;f(y)\,dy\\
&=&  (1-|x|^2)_+^s\; \int_{{\mathcal{D}}_{1,x}} g^\sharp(x,y)\;f(y)\,dy\\
&&\qquad+c_2
\int_{{\mathcal{D}}_{2,x}}|x-y|^{2s-n} \;f(y)\,dy+ 
\int_{{\mathcal{D}}_{2,x}}H^\star(x,y)\;f(y)\,dy\\
&=& (1-|x|^2)_+^s\,\tilde u(x),
\end{eqnarray*}
with some suitable~$\tilde{u}\in C^{\infty} (\overline{B})$.
This shows that $u\in C^{m,\sigma} (\overline{B}) \cap C^\infty (B)$ and also,
recalling Corollary~\ref{TRE}, that~$u\in H^s(\R^n)$, 
which completes the proof of~\eqref{TPR5}, and in turn of Theorem~\ref{TH}.
\end{proof}

\section*{Acknowledgements} 
This work has been supported by a Postdoc-fellowship 
of the Alexander von Humboldt Foundation
for the first author.

Serena Dipierro\\
School of Mathematics and Statistics,
University of Melbourne, 813 Swanston St,\\ Parkville VIC 3010, Australia,\\ and\\
School of Mathematics and Statistics,
University of Western Australia,\\
35 Stirling Highway,
Crawley, Perth
WA 6009, Australia\\
e--mail:  sdipierro@unimelb.edu.au \par\noindent
Hans--Christoph Grunau\\
Fakult\"at f\"ur Mathematik, Otto--von--Guericke--Universit\"at, Postfach 4120,\\
39016 Magdeburg, Germany\\
e--mail: hans-christoph.grunau@ovgu.de

\end{document}